\documentclass[a4paper,11pt]{article}
\usepackage[T1]{fontenc}
\usepackage{rotating}
\usepackage[title]{appendix}
\usepackage{amsmath} 
\usepackage{amsthm}
\usepackage{amsfonts}
\usepackage{mathtools}
\usepackage{amssymb}
\usepackage{enumitem}
\usepackage{gensymb}
\usepackage{subcaption}

\def\x{\mathbf{x}}

\DeclareMathOperator{\Circ}{Circ}
\DeclareMathOperator{\Cay}{Cay}

\DeclareMathOperator{\cart}{\Box}

\usepackage{graphicx}
\usepackage{url}
\usepackage[hyperfootnotes=true,colorlinks=true]{hyperref}
\hypersetup{
    colorlinks=true,
    linkcolor=blue,
    citecolor=magenta,      
    urlcolor=black
}

\usepackage{tikz}
\usetikzlibrary{calc}
\usetikzlibrary{backgrounds}
\usetikzlibrary{intersections}

\usepackage{authblk}
\usepackage{subcaption}
 
\usepackage[text={170mm,254mm}]{geometry}

\theoremstyle{plain}
\newtheorem{theorem}{Theorem}
\newtheorem{lemma}[theorem]{Lemma}

\newtheorem{proposition}[theorem]{Proposition}
\newtheorem{conjecture}[theorem]{Conjecture}

\theoremstyle{definition}

\newtheorem{problem}[theorem]{Problem}

\newtheorem*{remark}{Remark}
\newtheorem{example}[theorem]{Example}

\usepackage{amsfonts}
\usepackage{amsmath}
\usepackage[T1]{fontenc}
\usepackage[utf8]{inputenc}
\usepackage{lmodern}

\title{On the degrees of regular nut graphs and Cayley nut graphs}

\author[1,2,3]{Nino Ba{\v s}i{\'c}}
\author[1,4,5]{Ivan Damnjanovi{\'c}}
\author[6]{Patrick~W.~Fowler}
 
\affil[1]{FAMNIT, University of Primorska, Koper, Slovenia}
\affil[2]{IAM, University of Primorska, Koper, Slovenia}
\affil[3]{Institute of Mathematics, Physics and Mechanics, Ljubljana, Slovenia}
\affil[4]{Faculty of Electronic Engineering, University of Niš, Niš, Serbia}
\affil[5]{Diffine LLC, San Diego, California, USA}
\affil[6]{Chemistry, School of Mathematical and Physical Sciences, \linebreak University of Sheffield, Sheffield S3 7HF, UK}

\begin{document}

\maketitle

\begin{abstract}
A \emph{nut graph} is a simple graph for which the adjacency matrix has a 
single zero eigenvalue such that all non-zero kernel eigenvectors have no 
zero entry. 
It is known that infinitely many $d$-regular nut graphs exist for $3 \leq d \leq 12$ and for $d \geq 4$ such that $d \equiv 0 \pmod{4}$.
Here it is shown that infinitely many $d$-regular nut graphs exist for each degree $d \geq 3$.
Moreover, we prove that there are infinitely many $d$-regular Cayley nut graphs for each even $d \ge 4$.
This implies that we have identified all feasible degrees $d$ for which a $d$-regular Cayley nut graph exists.

\vspace{\baselineskip}

\vspace{\baselineskip}
\noindent
\textbf{Math.\ Subj.\ Class.\ (2020):} 
05C25, 
05C50. 
\end{abstract}

\section{Introduction}

\emph{Nut graphs}  are graphs that have a one-dimensional nullspace (i.e.\ $\eta(G) = 1$), 
where the non-trivial kernel eigenvector $\x = [x_1\ \cdots\ x_n]^\intercal \in \ker \mathbf{A}(G)$ is full (i.e.\ it has no zero entries). 
Early work established that \emph{non-trivial} nut graphs are connected, non-bipartite and have no
leaves, and showed that the smallest non-trivial nut graphs have seven or more vertices~\cite{ScirihaGutman-NutExt}. 

If $G$ is a non-trivial \emph{$d$-regular} nut graph, then the degree $d \geq 3$.
Regular nut graphs with $d = 2$ do not exist, as the nullity of a cycle is either $0$ or $2$.
The case $d = 3$ is of special interest in applications in chemistry and physics, where the simplest models for the electronic structure and properties of unsaturated carbon and hydrocarbon frameworks are essentially graph-theoretical \cite{Polansky1986,streitwieser1961}. Unsaturated molecules of this type are described by molecular graphs where the vertices represent carbon centres, and the eigenvectors and eigenvalues of the adjacency matrix are interpreted as $\pi$ molecular orbitals and orbital energies. In this context, the \emph{chemical graphs} are connected subcubic graphs, with $v_1$, $v_2$ and $v_3$ vertices of degree $1$, $2$ and $3$, respectively. Nut graphs have $v_1=0$. The pairs ($v_2$, $v_3$) for the \emph{chemical nut graphs} have been characterised \cite{Basic2020}. 

Chemical nut graphs are extremal in two properties of relevance to valence theory and the theory of ballistic conduction in molecular electronic devices. First, as the kernel eigenvector is full, these graphs correspond to unsaturated systems with fully distributed radical reactivity (i.e.\ with non-zero spin density on all carbon centres if the unique non-bonding $\pi$ orbital is singly occupied \cite{feuerpaper}). Secondly, the same graphs correspond to systems that would act as Fermi-level conductors, whenever connected by two leads in a simple electrical circuit; nut graphs are in fact exactly the \emph{strong omniconductors} of nullity $1$ \cite{PWF_JCP_2014,nuttybook}.

Cubic nut graphs are, therefore, relevant in applications to, for example, the polyhedral carbon cages known as fullerenes \cite{OmniFullerenes2013}.  The Frucht graph \cite{Frucht1939,Frucht1949} is a small example of a cubic polyhedral nut graph that would serve as a model of reactivity and conduction of a hypothetical $\text{C}_{12}$ cluster~\cite{feuerpaper}.

A recent development highlights an entirely different application of nut graphs in the theory of radio-frequency transmission in networks of coaxial cables.
Taking advantage of the match between kernels of the adjacency eigenvalue equation and Heaviside’s telegraph equations, it is possible to make an experimental simulation of the kernel eigenvector of a nut graph and verification of the strong omniconduction property \cite{McCarthy2025}. Commercial availability of three-way connectors makes the simulation of $3$-regular nut graphs by cable networks relatively straightforward.

From the mathematical perspective, there has been interest in the question of the possible degrees and orders of 
regular nut graphs of valency higher than three.
The main results of this paper are Theorems~\ref{thm:main} and~\ref{second_th} presented below where they are set in the context of the order--degree existence problem. These deal with, respectively, regular nut graphs in general and Cayley nut graphs.

\subsection{Regular nut graphs}

The order--degree existence problem for regular nut graphs was initiated by Gauci, Pisanski and Sciriha~\cite{GPS}. 
Let $\mathfrak{N}_d^{\mathrm{reg}}$ denote the set of all the orders attainable by a $d$-regular nut graph. In these terms, the problem is:

\begin{problem}[\hspace{1sp}{\cite[Problem~12]{GPS}}]\label{reg_problem}
    For each degree $d$, determine the set $\mathfrak{N}_d^{\mathrm{reg}}$.
\end{problem}

\noindent
In the initial paper, the following result was obtained:

\begin{theorem}[\hspace{1sp}{\cite[Theorems~2~and~3]{GPS}}]\label{reg_base_th_1}
The following holds:
\[
    \mathfrak{N}^{\mathrm{reg}}_3 = \{12\} \cup \{ n \in \mathbb{N} \colon \mbox{$n$ is even and } n \ge 18 \} \quad \mbox{and} \quad \mathfrak{N}^{\mathrm{reg}}_4 = \{ 8, 10, 12\} \cup  \{ n \in \mathbb{N} \colon n \ge 14 \} .
\]
\end{theorem}

\noindent
The proof of the above theorem relied on finding examples of smallest order and using a constuction to extend them
to all higher orders. This result was then extended as follows:

\begin{theorem}[\hspace{1sp}{\cite[Theorem~7]{Jan2020}}]
The following statements hold:
\begin{enumerate}[label=(\roman*)]\label{reg_base_th_2}
    \item $\mathfrak{N}^{\mathrm{reg}}_5 = \{ n \in \mathbb{N} \colon \mbox{$n$ is even and } n \ge 10 \}$;
    \item $\mathfrak{N}^{\mathrm{reg}}_6 = \{ n \in \mathbb{N} \colon n \ge 12 \}$;
    \item $\mathfrak{N}^{\mathrm{reg}}_7 = \{ n \in \mathbb{N} \colon \mbox{$n$ is even and } n \ge 12 \}$;
    \item $\mathfrak{N}^{\mathrm{reg}}_8 = \{ 12 \} \cup \{ n \in \mathbb{N} \colon n \ge 14 \}$;
    \item $\mathfrak{N}^{\mathrm{reg}}_9 = \{ n \in \mathbb{N} \colon \mbox{$n$ is even and } n \ge 16 \}$;
    \item $\mathfrak{N}^{\mathrm{reg}}_{10} = \{ n \in \mathbb{N} \colon n \ge 15 \}$;
    \item $\mathfrak{N}^{\mathrm{reg}}_{11} = \{ n \in \mathbb{N} \colon \mbox{$n$ is even and } n \ge 16 \}$.
\end{enumerate}
\end{theorem}

\noindent
The set $\mathfrak{N}^{\mathrm{reg}}_{12}$ was determined in \cite{basic2021regular}, where consideration of circulant graphs
furnished examples of such nut graphs of all even orders, while odd orders were resolved by a combination of computer search and construction.

\begin{theorem}[\hspace{1sp}{\cite[Theorem~1.3]{basic2021regular}}]\label{reg_base_th_3}
$\mathfrak{N}^{\mathrm{reg}}_{12} = \{ n \in \mathbb{N} \colon n \ge 16 \}$.
\end{theorem}

\noindent
Although the sets $\mathfrak{N}^{\mathrm{reg}}_{d}$ for $d > 12$ will not be completely determined, by providing families of
regular nut graphs of high degree, we show here that $\mathfrak{N}^{\mathrm{reg}}_{d}$ is (countably) infinite for every $d > 12$.
Again, circulant graphs play an essential role, as the families are defined by cartesian products, where one factor is a `small' regular nut graph and the
other a circulant graph with carefully chosen parameters. Our result is:

\begin{theorem}
\label{thm:main}
For each $d \ge 3$, there exist infinitely many $d$-regular nut graphs.
\end{theorem}

\subsection{Cayley nut graphs}

In parallel with the treatment of $d$-regular nut graphs in general, progress has been made on the order--degree existence problem for nut graphs
of high symmetry, in particular, those that belong to the classes of vertex-transitive and Cayley graphs. 

\emph{Vertex-transitive graphs} are graphs whose automorphism group acts transitively on the vertex set. Given a group $\Gamma$ with identity $e$ and a subset $C \subseteq \Gamma \setminus \{ e \}$ closed under inversion, by $\Cay(\Gamma, C)$ we denote the graph with vertex set $\Gamma$ such that any two vertices $u$ and $v$ are adjacent if and only if $v u^{-1} \in C$. The set $C$ is called the connection set. A \emph{Cayley graph} is any graph that is isomorphic to $\Cay(\Gamma, C)$ for some finite group $\Gamma$ and some connection set $C$. Every Cayley graph is vertex-transitive, while there exist vertex-transitive graphs that are not Cayley, perhaps the most famous example being the Petersen graph. A \emph{circulant graph} is a Cayley graph where $\Gamma \cong \mathbb{Z}_n$ for some $n \in \mathbb{N}$, i.e.\ the group $\Gamma$ is a cyclic group. For more details on algebraic graph theoretic terminology, the reader is referred to~\cite{Biggs,Dobson2022,Godsil2001}. For more results on highly symmetric nut
graphs, see \cite{BasicPWFFrucht,basic2023vertex,Zitnik2024} and references therein.

Let $\mathfrak{N}_d^{\mathrm{VT}}$ (resp.\ $\mathfrak{N}_d^{\mathrm{Cay}}$, $\mathfrak{N}_d^{\mathrm{circ}}$) be the set comprising the orders attainable by a $d$-regular vertex-transitive (resp.\ Cayley, circulant) nut graph. Trivially,
\[
    \mathfrak{N}_d^{\mathrm{circ}} \subseteq \mathfrak{N}_d^{\mathrm{Cay}} \subseteq \mathfrak{N}_d^{\mathrm{VT}} \subseteq \mathfrak{N}_d^{\mathrm{reg}}.
\]
The study of vertex-transitive nut graphs was initiated in \cite{Jan2020} with the next result.

\begin{theorem}[\hspace{1sp}{\cite[Theorem 10]{Jan2020}}]
\label{vt_base_th}
    Let $G$ be a non-trivial vertex-transitive nut graph on $n$ vertices, of degree $d$. Then $n$ and $d$ satisfy the following conditions. Either $d \equiv 0 \pmod 4$, and $n \equiv 0 \pmod 2$ and $n \ge d + 4$; or $d \equiv 2 \pmod 4$, and $n \equiv 0 \pmod 4$ and $n \ge d + 6$.
\end{theorem}

\noindent
Circulant nut graphs have been extensively studied \cite{Damnjanovic2022b, Damnjanovic2022a, Damnjanovic2022c, Damnjanovic2022}.
The paper \cite{Damnjanovic2022a} determined all the order--degree pairs for which circulant nut graphs exist.

\begin{theorem}[\hspace{1sp}{\cite[Theorem 1.8]{Damnjanovic2022a}}]\label{circ_base_th}
For each $d \in \mathbb{N}$, the set $\mathfrak{N}^{\mathrm{circ}}_d$ is given by
\[
    \mathfrak{N}^{\mathrm{circ}}_d = \begin{cases}
        \varnothing,& \mbox{if $d \not\equiv 0 \pmod 4$},\\
        \{ n \in \mathbb{N} \colon \mbox{$n$ is even and } n \ge d + 4 \},& \mbox{if $d \equiv 4 \pmod 8$},\\
        \{14\} \cup \{ n \in \mathbb{N} \colon \mbox{$n$ is even and } n \ge 18 \},& \mbox{if $d = 8$},\\
        \{ n \in \mathbb{N} \colon \mbox{$n$ is even and } n \ge d + 6 \},& \mbox{if $d \equiv 0 \pmod 8$ and $d \ge 16$}.
    \end{cases}
\]
\end{theorem}

\noindent
The order--degree existence problems for vertex-transitive and Cayley nut graphs were solved for degrees divisible by four.

\begin{theorem}[\hspace{1sp}{\cite[Corollaries 8 and 9]{DamnCayley}}]\label{cayley_base_th}
    For each $d \in \mathbb{N}$ such that $d \equiv 0 \pmod 4$, the sets $\mathfrak{N}^{\mathrm{VT}}_d$ and $\mathfrak{N}^{\mathrm{Cay}}_d$ are given by
    \[
        \mathfrak{N}^{\mathrm{VT}}_d = \mathfrak{N}^{\mathrm{Cay}}_d = \{ n \in \mathbb{N} \colon \mbox{$n$ is even and } n \ge d + 4 \} .
    \]
\end{theorem}

\noindent
Here, we give a counterpart to Theorem~\ref{thm:main}, tailored to the class of Cayley nut graphs.

\begin{theorem}\label{second_th}
    For each even $d \ge 4$, there exist infinitely many $d$-regular Cayley nut graphs.
\end{theorem}

\noindent
Note that Cayley nut graphs of odd degrees $d$ are already ruled out by Theorem~\ref{vt_base_th}. Earlier results of that kind relied on
an algorithm by Filaseta and Schinzel \cite{FilasetaSchinzel2004} for finding a cyclotomic factor in a polynomial with integer coefficients. In contrast, the proof of
Theorem~\ref{second_th} makes extensive use of trigonometric identities instead.

\section{Useful facts about graph spectra}

The \emph{spectrum} of a graph, denoted by $\sigma(G)$, is the spectrum of its adjacency matrix $A(G)$. 
For an in-depth treatment of graph spectra, see standard references \cite{haemers,Chung1997,Cvetkovic1997,Cvetkovic2010}.
The following lemma determines the spectrum (in closed form) of any circulant matrix.

\begin{lemma}\label{circ_lemma}
    For any $n \in \mathbb{N}$, the eigenvalues of a circulant matrix
    \[
    A = \begin{bmatrix}
        a_0 & a_1 & a_2 & \cdots & a_{n-1}\\
        a_{n-1} & a_0 & a_1 & \cdots & a_{n-2}\\
        a_{n-2} & a_{n-1} & a_0 & \cdots & a_{n-3}\\
        \vdots & \vdots & \vdots & \ddots & \vdots\\
        a_1 & a_2 & a_3 & \dots & a_0
    \end{bmatrix}
    \]
    are $P_A(\zeta)$, where
    \[
        P_A(x) = a_0 + a_1 x + a_2 x^2 + \cdots + a_{n-1} x^{n-1}
    \]
    and $\zeta \in \mathbb{C}$ ranges over the $n$-th roots of unity. Moreover, for each $n$-th root of unity $\zeta$, 
    \[
        \begin{bmatrix} 1 & \zeta & \zeta^2 & \cdots & \zeta^{n-1} \end{bmatrix}^\intercal
    \]
    is an eigenvector of $A$ for the eigenvalue $P_A(\zeta)$.
\end{lemma}

A proof of the above lemma can be found, e.g., in \cite[Section~3.1]{Gray2006}.
Note that the vertices of a circulant graph can be arranged so that the corresponding adjacency matrix is circulant. 
Therefore, Lemma \ref{circ_lemma} can be used to determine $\sigma(G)$ when $G$ is a \emph{circulant graph}. By $\Circ(n, S)$, 
where $S \subseteq \{1, 2, \ldots, \lfloor \frac{n}{2} \rfloor \}$, we will denote the graph on the 
vertex set $V(G) = \mathbb{Z}_n$, where vertices $u, v \in V(G)$ are adjacent if and only if $u - v  \in S$ or $v - u  \in S$ 
(where addition is done in $\mathbb{Z}_n$). Note that the adjacency matrix of $\Circ(n, S)$ is a circulant matrix with
$a_i = 1$ if $i \in S$ or $-i \in S$, and $a_i = 0$ otherwise. Thus, graph $\Circ(n, S)$ is a circulant graph
on $n$ vertices with connection set $S \cup -S$, i.e.\ $\Circ(n, S) \cong \Cay(\mathbb{Z}_n, S \cup -S)$. 

The \emph{cartesian product} of graphs $G$ and $H$ will be denoted by $G \cart H$. The vertex set of $G \cart H$
is $V(G) \times V(H)$; vertices $(g, h)$ and $(g', h')$  are adjacent if either
$g = g'$ and $hh' \in E(H)$, or $h = h'$ and $gg' \in E(G)$. 
For a comprehensive treatment of product graphs, see \cite{klavzar2011book},
and for a survey on their spectra, see \cite{Barik2018}. 

\begin{lemma}[{\cite[Section 1.4.6]{haemers}}]
\label{cart_lemma}
    For any two graphs $G$ and $H$, the spectrum of the cartesian product $G \cart H$ is given by
    \[
        \sigma(G \cart H) = \{ \lambda + \mu \colon \lambda \in \sigma(G), \, \mu \in \sigma(H) \} .
    \]
    Moreover, if $u \in \mathbb{R}^{V(G)}$ is an eigenvector of $G$ for the eigenvalue $\lambda$ and $v \in \mathbb{R}^{V(H)}$ is an eigenvector of $H$ for the eigenvalue $\mu$, then the vector $w \in \mathbb{R}^{V(G \cart H)}$ defined as
    \[
        w_{(g, h)} = u_g \, v_h \quad (g \in V(G),\ h \in V(H))
    \]
    is an eigenvector of $G \cart H$ for the eigenvalue $\lambda + \mu$.
\end{lemma}

\noindent
The following proposition is a direct consequence of Lemma~\ref{cart_lemma}.

\begin{proposition}
\label{prop:theProp}
Let $G_1$ and $G_2$ be nut graphs. Then $G_1 \cart G_2$ is a nut graph if  and only if there does not 
exist $\lambda \neq 0$ such that $\lambda \in \sigma(G_1)$ and $-\lambda \in \sigma(G_2)$.
\end{proposition}

For any $b \in \mathbb{N}$, the \emph{cyclotomic polynomial} $\Phi_b(x)$ is given by
\[
    \Phi_b(x) = \prod_{\zeta} (x - \zeta) ,
\]
where $\zeta$ ranges over the primitive $b$-th roots of unity. It is well known that for each $b \in \mathbb{N}$, the polynomial $\Phi_b(x)$ has integer coefficients and is irreducible in $\mathbb{Q}[x]$ (see, e.g., \cite[Chapter~33]{gallian}). Therefore, any $P(x) \in \mathbb{Q}[x]$ has a root that is a primitive $b$-th root of unity if and only if $\Phi_b(x) \mid P(x)$.

\section{Degrees of regular nut graphs}

In this section we will prove Theorem \ref{thm:main}. The following lemma will be useful.

\begin{lemma}\label{main_lemma}
For some $d \ge 3$, let $G$ be a $d$-regular nut graph, and for some $t \in \mathbb{N}$, let $S$ be a set consisting of $t$ odd and $t$ even positive integers. 
 If $d < 4t$, then there exists
$\ell \in \mathbb{N}$ such that for every prime number $p \ge \ell$, we have that $G \cart \Circ(2p, S)$ is a $(d + 4t)$-regular nut graph.
\end{lemma}
\begin{proof}
    Let $S = \{s_1, s_2, \ldots, s_{2t} \}$ and suppose that $n \in \mathbb{N}$ is even. 
For every $n \in \mathbb{N}, \, n \ge 2 \max S + 2$, let $H_n = \Circ(n, S)$.
By  Lemmas \ref{circ_lemma} and \ref{cart_lemma},  the eigenvalues of $G \cart H_n$ are $\nu(\lambda, \zeta)$, where
    \[
        \nu(z, x) = z + \sum_{j = 1}^{2t} \left( x^{s_j} + x^{-s_j} \right) , 
    \]
    and $\lambda$ ranges over $\sigma(G)$, while $\zeta$ ranges over the $n$-th roots of unity. Furthermore, we have
    \[
        \nu(\lambda, 1) = \lambda + 4t \quad \mbox{and} \quad \nu(\lambda, -1) = \lambda .
    \]
From the Perron-Frobenius theorem (see, e.g., \cite[Section 2.2]{haemers}) it follows that $|\lambda| \leq d$ for $\lambda \in \sigma(G)$, 
hence $\nu(\lambda, 1) > 0$. As the graph $G$ is a nut graph, it follows that $\nu(\lambda, -1) = 0$ is satisfied only for the simple eigenvalue $\lambda = 0$ of $G$. 
Since Lemmas \ref{circ_lemma} and \ref{cart_lemma} guarantee that the corresponding eigenvector is full, it follows that $G \cart H_n$ is a nut graph if and 
only if $\nu(\lambda, \zeta) \neq 0$ for every $\lambda \in \sigma(G)$ and every non-real $n$-th root of unity 
$\zeta \in \mathbb{C}$, $\zeta \notin \{ 1, -1 \}$.

Now, let $\alpha = \max S$. For each $\lambda \in \sigma(G)$, we see that $\nu(\lambda, \zeta) = 0$ holds if and only if $\zeta$ is a root of the $\mathbb{Z}[x]$-polynomial
    \[
        P_\lambda(x) = \lambda x^\alpha + \sum_{j = 1}^{2t} \left( x^{\alpha + s_j} + x^{\alpha - s_j} \right) .
    \]
    Furthermore, there are finitely many roots of unity that represent a root of $P_\lambda(x)$ for 
some $\lambda \in \sigma(G)$, 
and so we may take $\beta \in \mathbb{N}$  
to be the greatest of all of their orders. 
Note that $\beta$ is well-defined, as $-1$ certainly is a root of $P_0(x)$. 
If we let $\ell = 1 + \max(\alpha, \beta)$, then for each prime number $p \ge \ell$, every $2p$-th root of unity apart from $1$ and $-1$ has an order of 
either $p$ or $2p$, which means that it cannot be a root of 
any polynomial $P_\lambda(x), \, \lambda \in \sigma(G)$. 
Therefore, $G \cart H_{2p}$ is a nut graph for any $p \ge \ell$, and it is clear that the given graph is $(d + 4t)$-regular.
\end{proof}

To illustrate the steps in the above proof, we provide an example.

\begin{example}\label{example_e}
We show how to find examples of $17$-regular nut graphs using Lemma~\ref{main_lemma}. 
The values $d = 5$ and $t = 3$ fulfil the requirements of the lemma and 
will lead us to nut graphs of the desired degree. 
We require a $5$-regular nut graph. 
There are nine such graphs of order $10$ and four of order~$12$ (see, e.g., \cite{hog, hog2,CoolFowlGoed-2017,Jan2020,GPS}) and we could
choose any one of them. For example, let $G$ be the graph in Figure~\ref{fig:example5r}.
In the House of Graphs \cite{hog,hog2}, graph $G$ is listed last of the $5$-regular nut graphs on $10$ vertices. 

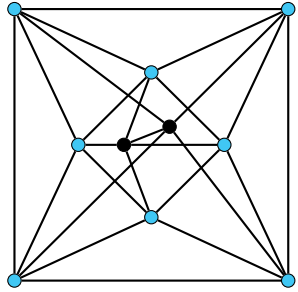
\begin{figure}[!h]
\centering
\begin{tikzpicture}[scale=1.2]
\tikzstyle{vertex}=[draw,circle,minimum size=5pt,inner sep=1pt,fill=cyan!60!white]
\tikzstyle{edge}=[draw,thick]
\node[vertex] (v7) at (0, 0) {};
\node[vertex] (v6) at (3, 0) {};
\node[vertex] (v5) at (0, -3) {};
\node[vertex] (v4) at (3, -3) {};
\node[vertex] (v3) at (1.5, -0.7) {};
\node[vertex] (v0) at (1.5, -2.3) {};
\node[vertex] (v2) at (0.7, -1.5) {};
\node[vertex] (v1) at (2.3, -1.5) {};
\node[vertex,fill=black] (v8) at (1.7, -1.3) {};
\node[vertex,fill=black] (v9) at (1.2, -1.5) {};
\path[edge] (v0) -- (v9) -- (v1);
\path[edge] (v2) -- (v9) -- (v3);
\path[edge] (v4) -- (v8) -- (v6);
\path[edge] (v5) -- (v8) -- (v7);
\path[edge] (v5) -- (v4) -- (v6) -- (v7) -- (v5);
\path[edge] (v0) -- (v1) -- (v3) -- (v2) -- (v0);
\path[edge] (v7) -- (v2) -- (v5) -- (v0) -- (v4) -- (v1) -- (v6) -- (v3) -- (v7);
\path[edge] (v9) -- (v8);
\end{tikzpicture}
\caption{The graph $G$ used in Example \ref{example_e}. Note that we obtain a square antiprism after deleting the two black vertices.}
\label{fig:example5r}
\end{figure}

The spectrum of $G$ is $\sigma(G) = \{ 5,(\sqrt{17}-1)/2, \sqrt{2}, \sqrt{2}, 0, -\sqrt{2}, -\sqrt{2}, -2, -2, -(\sqrt{17}+1)/2
\}$. 
Now we must choose a suitable set $S$. Any choice for which $S$ contains
$3$ odd and $3$ even integers is valid.
Let us take, for example, $S = \{1, 2, 3, 6, 7, 10\}$.
The corresponding value of 
the parameter $\alpha$ defined in the
proof is $\alpha = 10$, and the polynomial $P_\lambda(x)$ is then
\[
x^{20} + x^{17} + x^{16} + x^{13} + x^{12} + x^{11} + \lambda x^{10} + x^9 + x^8 + x^7 + x^4 + x^3 + 1.
\]
Let $\omega_n = e^{2\pi i / n}$ be a primitive $n$-th root of unity. By using a computer algebra system
such as \texttt{SageMath} \cite{SageMath}, we can verify that roots of unity appear as roots of the polynomial 
$P_\lambda(x)$ for certain values of $\lambda \in \sigma(G)$.
For $\lambda = \sqrt{2}$, such roots are $\omega_{8}^{3}$ and $\omega_{8}^{5}$, both of order~$8$;
for $\lambda = -\sqrt{2}$, they are $\omega_{8}$ and $\omega_{8}^{7}$, also both of order~$8$.
Finally, for $\lambda = 0$, the roots of unity that are roots of the polynomial $P_\lambda(x)$ are
\[
\omega_{2} = -1,\  \omega_{3},\  \omega_{3}^{2},\  \omega_{6},\  \omega_{6}^{5},\  \omega_{12},\  \omega_{12}^{5},\  \omega_{12}^{7},\  \omega_{12}^{11},\  
\omega_{18},\  \omega_{18}^{5},\  \omega_{18}^{7},\  \omega_{18}^{11},\   \omega_{18}^{13} \text{ and } \omega_{18}^{17}.
\]
The
orders of the roots in this list
are all between $2$ and $18$,
leading to $\beta = 18$. Hence, $\ell = 1 + \max(10, 18) = 19$.
For every prime number $p \geq 19$, the graph $G \cart \Circ(2p, \{1, 2, 3, 6, 7, 10\})$ is a $17$-regular nut graph. In particular,
$G \cart \Circ(38, \{1, 2, 3, 6, 7, 10\})$ is a $17$-regular nut graph of order~$380$. 

~\hfill $\Diamond$
\end{example}

Note that we have extensive freedom in choosing suitable $(d, t)$ pairs, choosing the graph $G$ and choosing the set $S$. 
In fact, in our example, as we were looking for a $17$-regular nut graph, the only possible choice for $(d, t)$ was $(5, 3)$. 
However, if we 
had been seeking, say, a $37$-regular nut graph, then we could have picked any $(d, t) \in \{(5, 8), (9, 7), (13, 6), (17, 5)\}$. 
There 
are 
always
infinitely many choices for the graph $G$. For the example, we chose the graph with the highest symmetry among those
 of the smallest order. Finally, there are  infinitely many feasible candidates for the set $S$. Needless to say, the value of $\ell$ depends on 
all of these
choices.

Note also that Lemma~\ref{main_lemma} gives us an infinite family of $(d+4t)$-regular nut graphs since we may pick any prime number $p \geq \ell$. However, it would be sufficient
to provide only one such graph, since a
construction described in \cite{GPS}, which 
replaces a vertex of degree $d$
of a nut graph by a gadget of $2d+1$ vertices of degree $d$ 
while retaining the nut property,
could have been used iteratively to produce an infinite family. 

With the aid of Lemma~\ref{main_lemma}, 
we can now prove Theorem~\ref{thm:main}.

\begin{proof}[Proof of Theorem~\ref{thm:main}]
Theorems \ref{reg_base_th_1}, \ref{reg_base_th_2}, \ref{reg_base_th_3} and \ref{circ_base_th} show that there exist infinitely many $d$-regular nut graphs for each $d \ge 3$ such that $d \le 12$ or $d \equiv 0 \pmod 4$. Thus, it is sufficient to show the existence of infinitely many $d$-regular
 nut graphs for every $d \ge 13$ such that $d \not\equiv 0 \pmod 4$.
    
Let $G_3$, $G_5$, and $G_6$ be 
arbitrary $3$-regular, $5$-regular, and $6$-regular nut graphs, respectively. 
Such graphs exist in abundance (see, e.g., \cite{Jan2020}). 
Let $S$ be any set consisting of $t$ odd and $t$ even positive integers, $t \ge 2$. Furthermore, for every $n \ge 2 \max S + 2$, let $H_n = \Circ(n, S)$. From Lemma~\ref{main_lemma}, we 
obtain immediately that each of $G_3 \cart H_n, G_5 \cart H_n$, and $G_6 \cart H_n$ is a nut graph for 
infinitely many values of $n \in \mathbb{N}$. Therefore, there exist infinitely many $(4t + 3)$-regular, $(4t + 5)$-regular and $(4t + 6)$-regular nut graphs, for any $t \ge 2$.
\end{proof}

\section{Degrees of Cayley nut graphs}

In the present section we prove Theorem \ref{second_th}. We begin with the following lemma.

\begin{lemma}\label{caux_lemma}
Let $\psi \in \mathbb{C}$ be a root of unity of order $b$ such that $b \ge 3$ and $b \not\equiv 0 \pmod 4$. Then $\psi$ is not a root of either of 
the polynomials
\[
    P_1(x) \coloneqq 2 x^{4t + 3} + x^{2t + 8} - x^{2t} - 2 x^5 \quad \mbox{and} \quad P_2(x) \coloneqq 2 x^{4t + 3} - x^{2t + 8} + x^{2t} - 2 x^5
\]
for any $t \in \mathbb{N}$.
\end{lemma}
\begin{proof}
By way of contradiction, suppose that $\psi$ is a root of $P_1(x)$ or $P_2(x)$ for some $t \in \mathbb{N}$. In this case, we have $\Phi_b(x) \mid P_1(x)$ or $\Phi_b(x) \mid P_2(x)$, hence $\omega \coloneqq e^{2 \pi i / b}$ is a root of $P_1(x)$ or $P_2(x)$. Therefore,
\begin{equation}\label{caux_3}
    2 \omega^{2t - 1} - 2 \omega^{-(2t - 1)} = \omega^{-4} - \omega^4 \quad \mbox{or} \quad     2 \omega^{2t - 1} - 2 \omega^{-(2t - 1)} = \omega^{4} - \omega^{-4} .
\end{equation}
By applying the identity $\frac{e^{i \theta} - e^{- i \theta}}{2i} = \sin \theta$ to \eqref{caux_3}, we obtain
\[
    2 \sin\tfrac{2(2t - 1) \pi}{b} = - \sin \tfrac{8 \pi}{b} \quad \text{or} \quad 2 \sin \tfrac{2(2t - 1) \pi}{b} = \sin \tfrac{8 \pi}{b} .
\]
Since $\sin\frac{8 \pi}{b} = 2 \sin \frac{4 \pi}{b} \cos \frac{4 \pi}{b}$, it follows that
\begin{equation}\label{caux_4}
    \left| \sin\tfrac{2(2t - 1) \pi}{b} \right| = \left| \sin \tfrac{4 \pi}{b} \right| \cdot \left| \cos \tfrac{4 \pi}{b} \right| .
\end{equation}
It is easy to verify, e.g., by using \texttt{SageMath}, that \eqref{caux_4} does not hold for any $b \in \{3, 4, \ldots, 8\}$ regardless of the value of $(2t - 1) \bmod b$. Therefore, $b \ge 9$. With this in mind, we have $\frac{4\pi}{b} \in (0, \tfrac{\pi}{2})$, which implies $\sin \frac{4 \pi}{b}, \cos \frac{4 \pi}{b} \in (0, 1)$. Hence, \eqref{caux_4} gives
\begin{equation}\label{caux_5}
    \left| \sin\tfrac{2(2t - 1) \pi}{b} \right| < \left| \sin \tfrac{4 \pi}{b} \right| .
\end{equation}
We proceed by splitting the argument into two cases depending on the parity of $b$.

\medskip\noindent
\textbf{Case 1:} $b$ is even.

Observe that if
\[
    (2t - 1) \bmod b \in \{ 3, 5, \ldots, \tfrac{b}{2} - 2\} \quad \mbox{or} \quad (2t - 1) \bmod b \in \{ \tfrac{b}{2} + 2, \tfrac{b}{2} + 4, \ldots, b - 3\},
\]
then $\left| \sin \frac{2(2t - 1) \pi}{b} \right| \ge \left| \sin \frac{4 \pi}{b} \right|$, which contradicts \eqref{caux_5}. Also, if $(2t - 1) \bmod b = \frac{b}{2}$, then $\sin \frac{2(2t - 1) \pi}{b} = 0$, which is impossible because $\sin \frac{4 \pi}{b}, \cos \frac{4 \pi}{b} \in (0, 1)$. Therefore, $(2t - 1) \bmod b \in \{1, b - 1\}$. With this in mind, \eqref{caux_4} becomes
\[
    \sin \tfrac{2 \pi}{b} = \sin \tfrac{4 \pi}{b} \cos \tfrac{4 \pi}{b}, \quad \text{i.e.} \quad 2 \cos \tfrac{2 \pi}{b} \cos \tfrac{4 \pi}{b} = 1.
\]
If $b = 10$, then
\[
    2 \cos \tfrac{2 \pi}{b} \cos \tfrac{4 \pi}{b} = 2 \cos\tfrac{\pi}{5} \cos \tfrac{2 \pi}{5} = \frac{\sin \frac{4 \pi}{5}}{2 \sin \frac{\pi}{5}} = \tfrac{1}{2},
\]
while for $b \ge 14$, we get
\[
    2 \cos \tfrac{2 \pi}{b} \cos \tfrac{4 \pi}{b} \ge 2 \cos \tfrac{\pi}{7} \cos \tfrac{2 \pi}{7} > 1,
\]
yielding a contradiction either way.

\medskip\noindent
\textbf{Case 2:} $b$ is odd.

Similarly, we observe that if
\[
    (2t - 1) \bmod b \in \{ 2, 3, 4, \ldots, \tfrac{b - 5}{2}\} \quad \mbox{or} \quad (2t - 1) \bmod b \in \{ \tfrac{b + 5}{2}, \tfrac{b + 7}{2}, \ldots, b - 2\},
\]
then $\left| \sin \frac{2(2t - 1) \pi}{b} \right| \ge \left| \sin \frac{4 \pi}{b} \right|$, which contradicts \eqref{caux_5}. Furthermore, if $(2t - 1) \bmod b = 0$, then $\sin \frac{2(2t - 1) \pi}{b} = 0$, which is not possible. Therefore, $(2t - 1) \bmod b \in \{1, \frac{b - 3}{2}, \frac{b - 1}{2}, \frac{b + 1}{2}, \frac{b + 3}{2}, b - 1\}$. If $(2t - 1) \bmod b \in \{1, b - 1\}$, then the proof can be completed analogously to Case 1.

Now, suppose that $(2t - 1) \bmod b \in \{ \frac{b - 1}{2}, \frac{b + 1}{2} \}$. In this case, \eqref{caux_4} transforms to
\begin{equation}\label{caux_6}
    \sin \tfrac{\pi}{b} = \sin \tfrac{4 \pi}{b} \cos \tfrac{4 \pi}{b} .
\end{equation}
By applying the identity $\sin 4 \theta = 4 \sin \theta \cos \theta \cos 2 \theta$ to \eqref{caux_6}, we get
\[
    4 \cos \tfrac{\pi}{b} \cos \tfrac{2 \pi}{b} \cos \tfrac{4 \pi}{b} = 1 .
\]
If $b = 9$, then
\[
    4 \cos \tfrac{\pi}{b} \cos \tfrac{2 \pi}{b} \cos \tfrac{4 \pi}{b} = 4 \cos \tfrac{\pi}{9} \cos \tfrac{2 \pi}{9} \cos \tfrac{4 \pi}{9} = \frac{\sin \frac{8\pi}{9}}{2 \sin \frac{\pi}{9}} = \tfrac{1}{2},
\]
while for $b \ge 11$, we have
\[
    4 \cos \tfrac{\pi}{b} \cos \tfrac{2 \pi}{b} \cos \tfrac{4 \pi}{b} \ge 4 \cos \tfrac{\pi}{11} \cos \tfrac{2 \pi}{11} \cos \tfrac{4 \pi}{11} > 1,
\]
which yields a contradiction either way.

Finally, suppose that $(2t - 1) \bmod b \in \{ \frac{b - 3}{2}, \frac{b + 3}{2} \}$. In this case, \eqref{caux_4} becomes
\begin{equation}\label{caux_7}
    \sin \tfrac{3 \pi}{b} = \sin \tfrac{4 \pi}{b} \cos \tfrac{4 \pi}{b}.
\end{equation}
By plugging the identities $\sin 3 \theta = \sin \theta \left( 2 \cos 2 \theta + 1 \right)$ and $\sin 4 \theta = 4 \sin \theta \cos \theta \cos 2 \theta$ into \eqref{caux_7}, we get
\[
    2 \cos \tfrac{2\pi}{b} + 1 = 4 \cos \tfrac{\pi}{b} \cos \tfrac{2\pi}{b} \cos \tfrac{4\pi}{b}, \quad \text{i.e.} \quad 2 \cos \tfrac{2\pi}{b} \left( 2 \cos \tfrac{\pi}{b} \cos \tfrac{4\pi}{b} - 1 \right) = 1 .
\]
It is easy to verify, e.g., by using \texttt{SageMath}, that
\[
    2 \cos \tfrac{2\pi}{b} \left( 2 \cos \tfrac{\pi}{b} \cos \tfrac{4\pi}{b} - 1 \right) < 1
\]
holds when $b \le 17$. Since $2 \cos \frac{\pi}{19} \cos \frac{4\pi}{19} > 1$, this means that for any $b \ge 19$, we have
\[
     2 \cos \tfrac{2\pi}{b} \left( 2 \cos \tfrac{\pi}{b} \cos \tfrac{4\pi}{b} - 1 \right) \ge 2 \cos \tfrac{2\pi}{19} \left( 2 \cos \tfrac{\pi}{19} \cos \tfrac{4\pi}{19} - 1 \right) > 1 ,
\]
yielding a contradiction.
\end{proof}

We are now in a position to complete the proof of Theorem \ref{second_th} by relying on the next lemma.

\begin{lemma}\label{cayley_lemma}
    For any $t \in \mathbb{N}$ and $m \ge 4t + 6$ such that $m \equiv 2 \pmod{4}$, the graph
    \[
        \Circ(m, \{1, 2, \ldots, t - 1\} \cup \{ \tfrac{m + 2}{4}, \tfrac{m + 6}{4} \} \cup \{\tfrac{m}{2} - (t - 1), \ldots, \tfrac{m}{2} - 2, \tfrac{m}{2} - 1\} \cup \{ \tfrac{m}{2} \}) \cart K_2
    \]
    is a $(4t + 2)$-regular Cayley nut graph of order $2m$.
\end{lemma}
\begin{proof}
Let
\[
    G \coloneqq \Circ(m, \{1, 2, \ldots, t - 1\} \cup \{ \tfrac{m + 2}{4}, \tfrac{m + 6}{4} \} \cup \{\tfrac{m}{2} - (t - 1), \ldots, \tfrac{m}{2} - 2, \tfrac{m}{2} - 1\} \cup \{ \tfrac{m}{2} \}) .
\]
Note that $t - 1 < \tfrac{m + 2}{4}$ and $\tfrac{m + 6}{4} < \tfrac{m}{2} - (t - 1)$, hence $G$ and $G \cart K_2$ are $(4t + 1)$- and $(4t + 2)$-regular, respectively. We also observe that $G \cart K_2$ is a Cayley graph for the group $\mathbb{Z}_m \times \mathbb{Z}_2$. By Lemmas \ref{circ_lemma} and \ref{cart_lemma}, the spectrum of $G \cart K_2$ comprises the values $\nu(\lambda, \zeta)$, where
\[
    \nu(\lambda, \zeta) = \lambda + \left( \zeta^{\frac{m+2}{4}} + \zeta^{-\frac{m+2}{4}} \right) + \left( \zeta^{\frac{m+6}{4}} + \zeta^{-\frac{m+6}{4}} \right) + \zeta^\frac{m}{2} + \sum_{j = 1}^{t - 1} \left( \zeta^j + \zeta^{-j} \right) + \sum_{j = \frac{m}{2} - t + 1}^{\frac{m}{2} - 1} \left( \zeta^j + \zeta^{-j} \right),
\]
and $\lambda \in \{1, -1\}$, while $\zeta$ ranges over the $m$-th roots of unity. A straightforward computation gives
\[
    \nu(1, 1) = 4t + 2, \quad \nu(-1, 1) = 4t, \quad \nu(1, -1) = 0 \quad \mbox{and} \quad  \nu(-1, -1) = -2 .
\]
Furthermore, Lemma \ref{cart_lemma} implies that the eigenvector for the eigenvalue $\nu(1, -1) = 0$ is full. Therefore, $G \cart K_2$ is a nut graph if and only if $\nu(\lambda, \zeta) \neq 0$ for any $\lambda \in \{1, -1 \}$ and $m$-th root of unity $\zeta \notin \{1, -1\}$. Observe that
\begin{equation}\label{nu_exp}
    \nu(\lambda, \zeta) = \lambda + \left( \zeta^{\frac{m+6}{4}} + \zeta^{\frac{m+2}{4}} +  \zeta^\frac{m}{2} + \zeta^{-\frac{m+2}{4}} + \zeta^{-\frac{m+6}{4}} \right) + \sum_{j = 1}^{t - 1} \left( \zeta^j + \zeta^{-j} + \zeta^{\frac{m}{2} - j} + \zeta^{j - \frac{m}{2}} \right) .
\end{equation}
We now split the argument into two cases depending on the value of $\zeta^\frac{m}{2}$.

\medskip\noindent
\textbf{Case 1:} $\zeta^\frac{m}{2} = -1$.

In this case we have $\zeta^{\frac{m}{2} - j} = -\zeta^{-j}$ and $\zeta^{j - \frac{m}{2}} = - \zeta^j$, hence \eqref{nu_exp} simplifies to
\[
    \nu(\lambda, \zeta) = \lambda - 1 + \zeta^{\frac{m+6}{4}} + \zeta^{\frac{m+2}{4}} + \zeta^{-\frac{m+2}{4}} + \zeta^{-\frac{m+6}{4}} .
\]
If $\lambda = 1$, then
\begin{align*}
    \nu(1, \zeta) &= \zeta^{\frac{m+6}{4}} + \zeta^{\frac{m+2}{4}} + \zeta^{-\frac{m+2}{4}} + \zeta^{-\frac{m+6}{4}} = \zeta^{-\frac{m+6}{4}} \left( \zeta^{\frac{m}{2} + 3} + \zeta^{\frac{m}{2} + 2} + \zeta + 1 \right)\\
    &= \zeta^{-\frac{m+6}{4}} \left( -\zeta^3 - \zeta^2 + \zeta + 1 \right) = -\zeta^{-\frac{m+6}{4}} (\zeta - 1) (\zeta + 1)^2 ,
\end{align*}
which implies $\nu(1, \zeta) \neq 0$ provided $\zeta \notin \{ 1, -1 \}$. Now, for contradiction, suppose that $\nu(-1, \zeta) = 0$ holds for some $\zeta \notin \{ 1, -1 \}$. Since
\begin{align*}
    \nu(-1, \zeta) = -2 + \zeta^{\frac{m+6}{4}} + \zeta^{\frac{m+2}{4}} + \zeta^{-\frac{m+2}{4}} + \zeta^{-\frac{m+6}{4}} = -2 -\zeta^{-\frac{m+6}{4}} (\zeta - 1) (\zeta + 1)^2,
\end{align*}
we obtain
\[
    (\zeta - 1)^2 (\zeta + 1)^4 = \left( -2 \zeta^{\frac{m + 6}{4}} \right)^2 = 4 \zeta^{\frac{m}{2} + 3} = -4 \zeta^3 .
\]
Therefore, $\zeta$ is a root of the polynomial
\begin{equation}\label{caux_1}
    (x - 1)^2 (x + 1)^4 + 4x^3 = x^6 + 2x^5 - x^4 - x^2 + 2x + 1 .
\end{equation}
It is easy to verify that \eqref{caux_1} is not divisible by $\Phi_b(x)$ for any $b \in \mathbb{N}$, yielding a contradiction.

\medskip\noindent
\textbf{Case 2:} $\zeta^\frac{m}{2} = 1$.

In this case, $\zeta^{\frac{m}{2} - j} = \zeta^{-j}$ and $\zeta^{j - \frac{m}{2}} = \zeta^j$, hence \eqref{nu_exp} transforms to
\begin{align}
    \nu(\lambda, \zeta) = \lambda + 1 + \zeta^{\frac{m+6}{4}} + \zeta^{\frac{m+2}{4}} + \zeta^{-\frac{m+2}{4}} + \zeta^{-\frac{m+6}{4}} + 2 \sum_{j = 1}^{t - 1} \left( \zeta^j + \zeta^{-j} \right) . \label{eq:burek_1}
\end{align}
As shown in \cite[Section~4]{Damnjanovic2022b}, we have $\nu(-1, \zeta) \neq 0$ for any $\frac{m}{2}$-th root of unity $\zeta \notin \{ 1, -1 \}$. Now, for contradiction, suppose that $\nu(1, \zeta) = 0$ holds for some $\zeta \notin \{ 1, -1 \}$, and note that \eqref{eq:burek_1} becomes
\[
    \nu(1, \zeta) = \zeta^{\frac{m+6}{4}} + \zeta^{\frac{m+2}{4}} + \zeta^{-\frac{m+2}{4}} + \zeta^{-\frac{m+6}{4}} + 2 \sum_{j = -t + 1}^{t - 1} \zeta^j .
\]
A routine computation gives
\begin{equation}\label{caux_2}
    \zeta^{t - 1} (\zeta - 1) \, \nu(1, \zeta) = \zeta^{t + \frac{m + 6}{4}} - \zeta^{t + \frac{m - 2}{4}} + \zeta^{t - \frac{m + 2}{4}} - \zeta^{t - \frac{m + 10}{4}} + 2 \zeta^{2t - 1} - 2 .
\end{equation}
Let $\psi \in \mathbb{C}$ be such that $\psi^2 = \zeta$, so that \eqref{caux_2} becomes
\begin{align*}
    \zeta^{t - 1} (\zeta - 1) \, \nu(1, \zeta) &= 2 \psi^{4t - 2} + \psi^{2t + 3 + \frac{m}{2}} - \psi^{2t - 1 + \frac{m}{2}} + \psi^{2t - 1 - \frac{m}{2}} - \psi^{2t - 5 - \frac{m}{2}} - 2\\
    &= 2 \psi^{4t - 2} + \psi^{2t + 3 + \frac{m}{2}} - \psi^{2t - 5 - \frac{m}{2}} - 2 .
\end{align*}
Since $m \equiv 2 \pmod 4$, we observe that $\psi$ is an $m$-th root of unity whose order $b$ satisfies $b \ge 3$ and $b \not\equiv 0 \pmod 4$. Depending on whether $\psi^{\frac{m}{2}} = 1$ or $\psi^{\frac{m}{2}} = -1$, we have 
\[
    2 \psi^{4t - 2} + \psi^{2t + 3} - \psi^{2t - 5} - 2 = 0 \quad \mbox{or} \quad 2 \psi^{4t - 2} - \psi^{2t + 3} + \psi^{2t - 5} - 2 = 0 .
\]
In both cases, a contradiction follows from Lemma \ref{caux_lemma}.
\end{proof}

The cases $d \equiv 0 \pmod 4$ and $d \equiv 2 \pmod 4$ of Theorem \ref{second_th} follow directly from Theorem \ref{cayley_base_th} and Lemma \ref{cayley_lemma}, respectively.

\begin{remark}
Note that Lemma \ref{cayley_lemma} also shows that
\[
    \mathfrak{N}^{\mathrm{VT}}_d \supseteq \mathfrak{N}^{\mathrm{Cay}}_d \supseteq \{ 2d + 8, 2d + 16, 2d + 24, 2d + 32, \ldots \}
\]
for any $d \ge 6$ such that $d \equiv 2 \pmod 4$. Therefore, for any even $d \ge 4$, the sets $\mathfrak{N}^{\mathrm{VT}}_d$ and $\mathfrak{N}^{\mathrm{Cay}}_d$ are not only infinite, but they also have a natural density of at least $1/8$.
\end{remark}

\section{Concluding remarks}

We have shown that $d$-regular nut graphs exist for all $d \geq 3$ (Theorem~\ref{thm:main}). When $d \equiv 0 \pmod{2}$, nut graphs exist even within the class of $d$-regular Cayley graphs (Theorem~\ref{second_th}).
In this sense, Theorem~\ref{second_th} is a stronger version of Theorem~\ref{thm:main}, but only for even degree.

The proof of Theorem \ref{second_th} relies on the nut graphs constructed in Theorem \ref{cayley_base_th} and Lemma \ref{cayley_lemma}, all of which are explicitly determined. However, the present proof of Theorem~\ref{thm:main} is non-constructive in nature. 
That is to say, the proof does not give an explicit algorithmic construction of the nut graphs 
in question, as the intermediate quantity $\beta$ (and hence $\ell$) is not easy to determine.
For even degree $d$, we could replace this branch of the proof by simple invocation of Theorem~\ref{second_th}, but for the odd case this route is not available. 
In the present section, we suggest a strategy that would lead to such a construction, based on a conjecture supported by numerical computation.

Recall that existence of $d$-regular nut graphs is guaranteed for $3 \leq d \leq 12$ (by Theorems~\ref{reg_base_th_1}, \ref{reg_base_th_2} and~\ref{reg_base_th_3}).
Recall also the following theorem from previous work.

\begin{theorem}[{\cite[Theorem 4]{Damnjanovic2022b}}]
For each $t \in \mathbb{N}$ and $n \geq 4t + 6$ such that $n \equiv 2 \pmod{4}$, the circulant graph
\[
\mathcal{D}(n, t) \coloneqq
\Circ(n, \{1, 2, \ldots, t - 1\} \cup \{(n + 2) / 4, (n + 6) / 4\} \cup \{n/2 - (t - 1), \ldots, n/2 - 2, n/2 - 1\})
\]
is a $4t$-regular nut graph of order $n$.
\end{theorem}

We can obtain nut graphs with odd degrees $d \geq 13$ using Proposition~\ref{prop:theProp} by taking cartesian products of
members of the above family of circulant graphs with certain $3$- and $5$-regular nut graphs. For example,
let $F_3$ be the Frucht graph, and let $F_5$ be the graph whose complement is shown in Figure~\ref{fig:2}. 
Their characteristic polynomials are:
\begin{equation*}
\begin{aligned}
\Phi(F_3, x) & = x (x - 3) (x - 2)  (x + 1)  (x + 2)  (x^3 + x^2 - 2x - 1)  (x^4 + x^3 - 6x^2 - 5x + 4),\\
\Phi(F_5, x) & = x (x - 5) (x + 1)  (x^2 + x - 1)  (x^5 + 3 x^4 - 6 x^3 - 21 x^2 - x + 16).
\end{aligned}
\label{eq:charPoly}
\end{equation*}

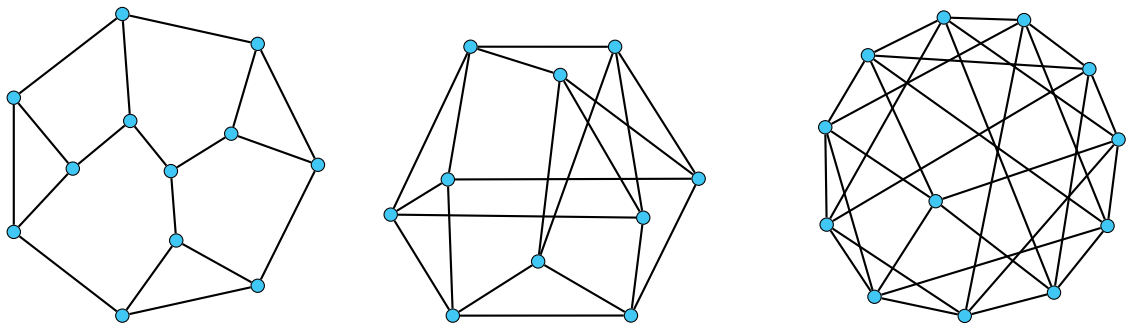
\begin{figure}[!htbp]
\centering
\subcaptionbox{The Frucht graph $F_3$.\label{fig:cubic3}}
{\begin{tikzpicture}[scale=0.5]
\tikzstyle{vertex}=[draw,circle,minimum size=5pt,inner sep=1pt,fill=cyan!60!white]
\tikzstyle{edge}=[draw,thick]
\node[vertex] (0) at (3.1428571428571423, 0.8571428571428563) {};
\node[vertex] (1) at (1.558957771787959, 4.064597491767673) {};
\node[vertex] (2) at (-2.0012938005390835, 4.857142857142858) {};
\node[vertex] (3) at (-4.857142857142858, 2.636446437329223) {};
\node[vertex] (4) at (-4.857142857142858, -0.9221607230435094) {};
\node[vertex] (5) at (-2.0012938005390835, -3.1428571428571423) {};
\node[vertex] (6) at (1.558957771787959, -2.3503117774819593) {};
\node[vertex] (7) at (0.8626055705301008, 1.6830379037343235) {};
\node[vertex] (8) at (-1.793998203054807, 2.020458207804947) {};
\node[vertex] (9) at (-3.3054447439353094, 0.7570938134940097) {};
\node[vertex] (10) at (-0.5844384546271328, -1.1497232536957893) {};
\node[vertex] (11) at (-0.7253189577717882, 0.690394451061445) {};
\path[edge] (0) -- (1);
\path[edge] (0) -- (6);
\path[edge] (0) -- (7);
\path[edge] (1) -- (2);
\path[edge] (1) -- (7);
\path[edge] (2) -- (3);
\path[edge] (2) -- (8);
\path[edge] (3) -- (4);
\path[edge] (3) -- (9);
\path[edge] (4) -- (5);
\path[edge] (4) -- (9);
\path[edge] (5) -- (6);
\path[edge] (5) -- (10);
\path[edge] (6) -- (10);
\path[edge] (7) -- (11);
\path[edge] (8) -- (9);
\path[edge] (8) -- (11);
\path[edge] (10) -- (11);
\end{tikzpicture}
}
\subcaptionbox{Complement of the graph $F_5$.\label{fig:quintic5}}
{\quad\begin{tikzpicture}[scale=0.5]
\tikzstyle{vertex}=[draw,circle,minimum size=5pt,inner sep=1pt,fill=cyan!60!white]
\tikzstyle{edge}=[draw,thick]
\node[vertex] (0) at (-0.3961461536029489, 4.037291192002961) {};
\node[vertex] (1) at (1.0444420489822148, 4.7837993431029595) {};
\node[vertex] (2) at (1.4619269635504972, -2.3438056441274444) {};
\node[vertex] (3) at (-3.3550074363408178, 1.2648856857001736) {};
\node[vertex] (4) at (-4.857142857142858, 0.33298875833595076) {};
\node[vertex] (5) at (-3.223623336112798, -2.3468527198346854) {};
\node[vertex] (6) at (-2.758147923799228, 4.785714285714286) {};
\node[vertex] (7) at (-0.978752746296677, -0.9102738864029369) {};
\node[vertex] (8) at (1.7857142857142863, 0.25110514429792374) {};
\node[vertex] (9) at (3.2420824470944978, 1.284744801791595) {};
\path[edge] (0) -- (6);
\path[edge] (0) -- (7);
\path[edge] (0) -- (8);
\path[edge] (0) -- (9);
\path[edge] (1) -- (6);
\path[edge] (1) -- (7);
\path[edge] (1) -- (8);
\path[edge] (1) -- (9);
\path[edge] (2) -- (5);
\path[edge] (2) -- (7);
\path[edge] (2) -- (8);
\path[edge] (2) -- (9);
\path[edge] (3) -- (4);
\path[edge] (3) -- (5);
\path[edge] (3) -- (6);
\path[edge] (3) -- (9);
\path[edge] (4) -- (5);
\path[edge] (4) -- (6);
\path[edge] (4) -- (8);
\path[edge] (5) -- (7);
\end{tikzpicture}\quad
}
\caption{The regular nut graphs 
used in Conjecture~\ref{conj}
(drawn as the graph or its complement).}
\label{fig:2}
\end{figure}

We now propose the following conjecture.

\begin{conjecture}
\label{conj}
The following statements hold:
\begin{enumerate}[label=(\roman*)]
\item
$\mathcal{D}(4t + 6, t) \cart F_5$ is a nut graph of degree $4t + 5$ for every $t \geq 1$;
\item
$\mathcal{D}(4t + 6, t) \cart F_3$ is a nut graph of degree $4t + 3$ for $t \geq 1$ such that $t \not\equiv 0 \pmod{3}$;
\item
$\mathcal{D}(4t + 10, t) \cart F_3$ is a nut graph of degree $4t + 3$ for $t \geq 1$ such that $t \equiv 0 \pmod{3}$.
\end{enumerate}
\end{conjecture}

Statement (i) resolves the case $d \equiv 1 \pmod{4}$.
Statements (ii) and (iii) taken together resolve the case $d \equiv 3 \pmod{4}$. 
Note that the graph $\mathcal{D}(4t + 6, t)$ has an eigenvalue $-3$ when $t \equiv 0 \pmod{3}$. 

Computations show that the conjecture is true for all values of $t$ such that the degree of the obtained product graph is less than $1000$.
This implies that we have constructed a $d$-regular nut graph for each odd $3 \leq d < 1000$.
Proof of Conjecture~\ref{conj} would furnish a fully constructive proof of Theorem~\ref{thm:main}.

\section*{Acknowledgements}
The work of Nino Bašić is supported in part by the Slovenian Research Agency (research program P1-0294
and research projects J5-4596 and J1-2481).
Ivan Damnjanović is supported by the Ministry of Science, Technological Development and Innovation of the Republic of Serbia, grant number 451-03-137/2025-03/200102, 
and the Science Fund of the Republic of Serbia, grant \#6767, Lazy walk counts and spectral radius of threshold graphs --- LZWK.
Patrick Fowler thanks the Leverhulme Trust for an
Emeritus Fellowship on 
{\lq Modelling molecular currents, conduction and aromaticity\rq} and the Francqui Foundation for award of
an International Francqui Professorship.

\end{document}